\documentclass[hha]{article}
\usepackage{amsmath}
\usepackage{amssymb}
\usepackage{amsthm}
\usepackage[mathscr]{eucal}
\usepackage[all,knot,cmtip]{xy}
\xyoption{arc}
\usepackage{soul}
\usepackage{graphicx}

\numberwithin{equation}{section}
\newtheorem{thm}{Theorem}[section]

\newtheorem{prop}[thm]{Proposition}
\newtheorem{lem}[thm]{Lemma}

\newcommand{\vmodels}{\rotatebox[origin=c]{-90}{$\models$}}
\newcommand{\modelsv}{\rotatebox[origin=c]{90}{$\models$}}

\newtheorem{cor}[thm]{Corollary}

\theoremstyle{definition}
\newtheorem{definition}[thm]{Definition}
\newtheorem{rem}[thm]{Remark}

\begin{document}
\title{On the stratification of combinatorial spectra}
\author{Ryo Horiuchi}
\date{}

\maketitle

\begin{abstract}
In this note, we investigate a mixture of combinatorial spectra and stratified simplicial sets, which would be thought of as a model of the spectrum objects of $(\infty, \infty)$-categories.
\end{abstract} 

\section{Introduction}

In this note, we investigate a mixture of combinatorial spectra and stratified simplicial sets, which would be thought of as a model of the spectrum objects of $(\infty, \infty)$-categories.

In \cite{Kan}, Kan introduced the notion of combinatorial spectrum, which is a certain presheaf in pointed sets over a category $\Delta_{st}$ that is a stabilization of the simplex category $\Delta$ in an appropriate sense.
Intuitively speaking, a combinatorial spectrum is a pointed simplicial set with $\mathbb{Z}$-graded simplices.
In \cite{KW}, Kan and Whitehead introduced a product for combinatorial spectra, which may not give rise to a monoidal structure on the nose, and showed that it works well up to homotopy.
Later, Brown showed the category of combinatorial spectra admits a model structure in \cite{B}.
Bousfield and Friedlander showed there exists a chain of Quillen equivalences between that and another model of spectra in \cite{BF}.

In \cite{OR}, Ozornova and Rovelli introduced the notion of prestratified simplicial sets, which has already appeared in \cite{V0} without name and is also a presheaf in sets over a category similar to $\Delta$, denoted by $t\Delta$.
Intuitively speaking, a prestratified simplicial set is a simplicial set with two layers of simplices.
Verity constructed in \cite{V1} a model structure on the category of stratified simplicial sets, which are prestratified simplicial sets satisfying a certain condition.
Ozornova and Rovelli then constructed a model structure on the category of prestratified simplicial sets based on Verity's work and it is expected the model structure models of $(\infty, \infty)$-categories.

In this note we investigate the presheaves in pointed sets over a stabilization of $t\Delta$.
Such a presheaf thus would be viewed as a pointed simplicial set consists of two kinds of $n$-simplices with $n\in\mathbb{Z}$.
We show the presheaf category  inherits a homotopy theory and the stratified analogue of Kan-Whitehead product is compatible with that in an appropriate sense.

\section{Preliminary}

\subsection{Stabilization of simplicial sets}
In this section, we recall combinatorial spectra from \cite{Kan}, \cite{KW} and \cite{B}.
We focus on Kan-Whitehead product and  Brown's model structure of them.

Let $\Delta$ denote the simplex category and $\star:\Delta\times\Delta\to\Delta$ denote the concatenation functor.
More precisely, for any $[m], [n]\in\Delta$, $[m]\star[n]=[m+n+1]$, and for $\theta:[m]\to[m'], \tau:[n]\to[n']\in\Delta$, the morphism $\theta\star\tau:[m+n+1]\to[m'+n'+1]$ is given by
\[(\theta\star\tau)(i) =
\begin{cases}
\theta(i) & 0 \leq i\leq m,\\
\tau(i-m-1)+m'+1 & m+1\leq i\leq m+m'+1
\end{cases}\]

\begin{rem}By definition we get the following equations:
\[d^i\star s^j=d^i\circ s^{n+j}, \ s^j\star d^i=d^{m+1+i}\circ s^j, d^i\star d^{i'}=d^{n+1+i'}\circ d^{i}, \ \    s^j\star s^{j'}=s^{m+1+j'}\circ s^{j}\ \]
for any face operators $d^i:[n-1]\to [n]$ and $d^{i'}:[n'-1]\to [n']$, and degeneracy operators $s^j:[m+1]\to [m]$ and $s^{j'}:[m'+1]\to [m']$.
\end{rem}

Using this concatenation functor, Chen, Kriz, and Pultr stabilized $\Delta$ in the following sense.

\begin{definition}
We define the following shift functors:

\[K:\Delta\to\Delta,\ K(\theta:[n]\to[m])=(\theta\star[0]:[n+1]\to[m+1]),\]

The other one $J$ is the dual of $K$.
\[J:\Delta\to\Delta,\ J(\theta:[n]\to[m])=([0]\star\theta:[n+1]\to[m+1]),\]

By using these functors we define the following categories:
\[\Delta_{st}:=\operatorname{colim}(\Delta\xrightarrow{K}\Delta\xrightarrow{K}\Delta\xrightarrow{K}\cdots)\]

\[\Delta_{st'}:=\operatorname{colim}(\Delta\xrightarrow{J}\Delta\xrightarrow{J}\Delta\xrightarrow{J}\cdots)\]

\[\Delta_{st^2}:=\operatorname{colim}(\Delta\xrightarrow{K}\Delta\xrightarrow{J}\Delta\xrightarrow{K}\Delta\xrightarrow{J}\Delta\xrightarrow{K}\cdots)\]

\end{definition}

\begin{rem}\label{identities}(1)The first one $\Delta_{st}$ is defined in \cite{CKP}.
By definition, we have $J\circ K=K\circ J$.
So we can change the order of the functors in the colimit diagram of $\Delta_{st^2}$.

(2)These three categories above admit the following descriptions respectively.
By abusing notation, the objects of $\Delta_{st}$ may be denoted by $[n]$ for all $n\in\mathbb{Z}$.
The morphism of $\Delta_{st}$ are generated by the morphisms
\[d^i:[n-1]\to[n], \ s^j:[n+1]\to[n]\]
for all $n\in\mathbb{Z}$ and integers $i, j\geq 0$ subject to the following identities
\[d^id^j=d^{j+1}d^i \ \  (i\leq j)\]
\[s^js^i=s^is^{j+1} \ \ (i\leq j)\]
\[s^jd^i= 
\begin{cases}
d^is^{j-1} & (i<j)\\
\operatorname{id} & (i\in\{j, j+1\})\\
d^{i-1}s^j &(i>j+1).
\end{cases}\]

Dually, the objects of $\Delta_{st'}$ are denoted by $[n]$ for all $n\in\mathbb{Z}$.
The morphism of $\Delta_{st'}$ are generated by the morphisms
\[d^i:[n-1]\to[n], \ s^j:[n+1]\to[n]\]
for all $n\in\mathbb{Z}$ and integers $i, j\leq n$ subject to the same identities.

The objects of $\Delta_{st^2}$ are denoted by $[n]$ for all $n\in\mathbb{Z}$.
The morphism of $\Delta_{st^2}$ are generated by the morphisms
\[d^i:[n-1]\to[n], \ s^j:[n+1]\to[n]\]
for all $n, i, j\in\mathbb{Z}$ subject to the same identities.

(3)There is a functor $\operatorname{rev}:\Delta\to\Delta$ defined by
\[\operatorname{rev}(d^i:[n-1]\to[n])=(d^{n-i}:[n-1]\to[n]),\]
\[\operatorname{rev}(s^i:[n+1]\to[n])=(s^{n-i}:[n+1]\to[n]).\]

This fits into the commutative diagram
 \[
   \xymatrix{
 \Delta\ar^{K}[r] \ar_{\operatorname{rev}}[d]& \Delta\ar^{\operatorname{rev}}[d]\\
 \Delta\ar_{J}[r]&\Delta.\\     
}
\]
This gives a functor $\Delta_{st}\to\Delta_{st'}$, which we also denote by $\operatorname{rev}$ abusing notation.
\end{rem}

\begin{definition}[\cite{Kan}, \cite{CKP}] We let $\mathbf{Set}_*$ denote the category of pointed sets and pointed maps.
\begin{enumerate}
\item A functor $\Delta_{st}^{op}\to \mathbf{Set}_*$ is called a stable simplicial set.
A stable simplicial set $X$ satisfying the following condition is called a combinatorial spectrum:
for every $x\in X$, there exists an integer $i$ such that $d_kx=*$ when $k>i$.
We denote by $\mathbf{Set}_*^{\Delta_{st}^{op}}$ the category of stable simplicial sets and natural transformations and by $\mathbf{Comb}$ the full subcategory of combinatorial spectra.

\item A functor $\Delta_{st^2}^{op}\to \mathbf{Set}_*$ is called a bistable simplicial set.
A bistable simplicial set $X$ satisfying the following condition is called a combinatorial bispectrum:
for every $x\in X$, there exist integers $i$ and $j$ such that $d_kx=*$ when $k>i$ or $k<j$.
We denote by $\mathbf{Set}_*^{\Delta_{st^2}^{op}}$ the category of bistable simplicial sets and natural transformations respectively.
\end{enumerate}
\end{definition}

We construct the stable analogue of the concatenation functor as follows:

Let $\theta:[n]\to[n']\in\Delta_{st'}$ and $\tau:[m]\to[m']\in\Delta_{st}$.
By the definitions of $\Delta_{st'}$ and $\Delta_{st}$, there exist $k, l\in\mathbb{N}$ and morphisms $\tilde{\theta}, \tilde{\tau}\in\Delta$ such that
\[\tilde{\theta}:[n+k]\to[n'+k]\in\Delta\]
represents $\theta$ and
\[\tilde{\tau}:[m+l]\to[m'+l]\in\Delta\]
represents $\tau$.
We denote $\theta\star\tau:[n]\star[m]\to[n']\star[m']$ the morphism in $\Delta_{st^2}$ represented by
\[\tilde{\theta}\star\tilde{\tau}:[n+m+k+l]\to[n'+m'+k+l]\in\Delta\]
This defines the functor again denoted by
\[\star:\Delta_{st'}\times\Delta_{st}\to\Delta_{st^2}.\]

For $X, Y\in\mathbf{Set}_*^{\Delta_{st}^{op}}$, we let $X\wedge'' Y$ denote the point-wise smash product, namely $(X\wedge'' Y)([n])=X([n])\wedge Y([n])$, where $\wedge$ denotes the smash product of pointed sets.

Let $X\wedge' Y$ denote the left Kan extension of $X\wedge''Y$ with respect to 
\[\Delta_{st}^{op}\times\Delta_{st}^{op}\xrightarrow{\operatorname{rev}\times \operatorname{id}}\Delta_{st'}^{op}\times\Delta_{st}^{op}\xrightarrow{\star}\Delta_{st^2}^{op}.\]

This construction defines a functor $\mathbf{Set}_*^{\Delta_{st}^{op}}\times \mathbf{Set}_*^{\Delta_{st}^{op}}\to\mathbf{Set}_*^{\Delta_{st^2}^{op}}$, which is the stable analogue of the join construction of simplicial sets.
To define the product on $\mathbf{Set}_*^{\Delta_{st}^{op}}$ called reduced join in \cite{KW}, we recall a functor $\mathbf{Set}_*^{\Delta_{st^2}^{op}}\to\mathbf{Set}_*^{\Delta_{st}^{op}}$.

\begin{definition}[\cite{Kan}]
Let $V\in\mathbf{Set}_*^{\Delta_{st^2}^{op}}$.
We define $V_{-1}\in\mathbf{Set}_*^{\Delta_{st}^{op}}$ as follows.
For any $[n]\in \Delta_{st}$, we set
\[V_{-1}([n]):=\{x\in V([n+1])|d^V_jx=*, j<1\}\]

The generators of maps are given as follows:
\[d_i:=d_{i+1}^V, s_i:=s_{i+1}^V,\]
where $d_{j}^V$ and $s_{j}^V$ are the generators for $V$.
\end{definition}

\begin{definition}[\cite{KW}]Let $X, Y\in\mathbf{Set}_*^{\Delta_{st}^{op}}$. The Kan-Whitehead smash product $X\wedge Y$ is $(X\wedge'Y)_{-1}$.
\end{definition}

\begin{rem}In \cite{KW} Kan and Whitehead took a functor before applying the functor $(\mathchar`-)_{-1}:\mathbf{Set}_*^{\Delta_{st^2}^{op}}\to\mathbf{Set}_*^{\Delta_{st}^{op}}.$
More precisely, by using freely generating functor $\operatorname{F}:\mathbf{Set}\to\mathbf{Grp}$ piecewise, where $\mathbf{Grp}$ denotes the category of groups and homomorphisms, they considered $(\operatorname{F}(X\wedge'Y))_{-1}$.
There the celebrated fact that the underlying simplicial set of any simplicial group is a Kan complex plays a pivotal role.

\end{rem}

Kan and Whitehead have shown this product plays a role in the stable homotopy theory.
To descrive that, we recall prespectra from \cite{Kan}.

As is pointed out in \cite{L}, Kan's suspension functor (\cite[Definition 2.2]{Kan}) $\operatorname{S}:\mathbf{sSet}\to \mathbf{sSet}_*$ is the left Kan extension along yoneda embedding of
\[\Delta\to \mathbf{sSet}_*, [n]\mapsto \Delta[n+1]_+/(\Delta[n]_+\wedge \Delta[0]_+).\]

\begin{definition}[\cite{Kan}]
(1) A prespectrum $L$ consists of

(i) a sequence of pointed simplicial sets $L_i$ with $i\in\mathbb{N}$,

(ii) a sequence of monomorphisms $\lambda_i:\operatorname{S}L_i\to L_{i+1}$ of pointed simplicial sets with $i\in\mathbb{N}$, where $\operatorname{S}$ denotes the suspension functor \cite[Definition 2.2]{Kan}.

(2) A morphism $\psi:\{L_i, \lambda_i\}\to \{M_i, \mu_i\}$ of prespectra is a sequence of morphisms $\psi_i:L_i\to M_i$ of pointed simplicial sets such that $\psi_{i+1}\circ\lambda_i=\mu_i\circ\operatorname{S}\psi_i$.

(3) A morphism $\psi:\{L_i, \lambda_i\}\to \{M_i, \mu_i\}$ of prespectra is called a weak homotopy equivalence if for every $q\in\mathbb{Z}$ the abelian group homomorphism
\[\pi_q(\psi):\operatorname{colim}_{i\to\infty}\pi_{i+q}(L_i)\to\operatorname{colim}_{i\to\infty}\pi_{i+q}(M_i)\]
is an isomorphism.
\end{definition}

Kan related combinatorial spectra and prespectra as follows.

\begin{definition}Let $X\in\mathbf{Set}_*^{\Delta_{st}^{op}}$.
We define the corresponding prespectrum $\operatorname{Ps}(X)=\{X_i, \xi_i\}$ as follows.

For any $i\in\mathbb{N}$, the pointed set of $n$-simplices in $X_i$ is given by
\[X_i([n]):=\{\alpha\in X([n-i])|d_0\cdots d_n\alpha=*, d_j\alpha=*,   (j\geq n)\},\]
where $*$ denotes (the degeneracy of) the base point.
Face and degeneracy operators on $X_i$ will be induced by those of $X$ and $X_i$ is indeed a pointed simplicial set with them.

For any $i\in\mathbb{N}$, the monomorphism $\xi_i:\operatorname{S}X_i\to X_{i+1}$ is given by $\xi_i(\alpha, \phi)=\alpha$.
\end{definition}

Brown has shown that the weak homotopy equivalences give rise to a model structure.
\begin{thm}[\cite{B}]The category of combinatorial spectra admits the model structure in which
\begin{itemize}
\item cofibrations are precisely monomorphims,
\item a morphism $f$ is a fibration if and only if $\operatorname{Ps}(f)_i$ is a fibration for all $i$ in the classical model structure on pointed simplicial sets,
\item weak equivalences are the morphisms whichare weak homotopy equivalences after taking $\operatorname{Ps}$.
\end{itemize}
\end{thm}

Kan and Whitehead has proven that the product is compatible with the weak equivalences.
\begin{thm}[\cite{KW}]
The functor
\[(\operatorname{F}((\mathchar`-)\wedge'(\mathchar`-)))_{-1}:\mathbf{Comb}\times\mathbf{Comb}\to\mathbf{Comb}\]
preserves weak equivalences.
\end{thm}

Let $\mathbb{S}$ be the stable simplicial set given by
\[(\mathbb{S})([n]) =
\begin{cases}
\{*, \alpha_n\} & 0 \leq n,\\
\{*\} & n<0
\end{cases}\]
In other words, this is the stable simplicial set corresponding to the simplicial set $\Delta[0]$.
It is shown there that the sphere spectrum acts as the unit of the product up to homotopy.

\begin{thm}[\cite{KW}] For any combinatorial spectrum $X$, $(\operatorname{F}(\mathbb{S}\wedge' X))_{-1}$ is weakly equivalent to $X$.
\end{thm}

\begin{rem}
In \cite[Section 10]{Kan} Kan defined the $q$-th homotopy group $\pi_q(X)$ of a combinatorial spectrum $X$ to be the homotopy group $\pi_{q+i}((\operatorname{F}X)_i)$ with an integer $i>-q$, where $F$ denotes the free group functor.
This is well-defined since $\pi_n((\operatorname{F}X)_j)=\pi_{n+1}((\operatorname{F}X)_{j+1})$ for all $n$ and $j$.
It is shown there that a morphism $f$ of combinatorial spectra is a weak equivalence if and only if the morphism $\operatorname{Ps}(f)$ of prespectra is a termwise weak homotopy equivalence.
\end{rem}

\subsection{Stratification of simplicial sets}
In this section, we recall  (pre-)stratified simplicial sets from mainly \cite{OR} and \cite{V1}.
We focus on the join construction, the lax Gray-Verity tensor product and Ozornova-Rovelli model structure of them.

We recall the category $t\Delta$ from \cite{OR}.
Its set of objects consists of $[n]$ with $0\leq n\in\mathbb{Z}$ and $[m]_t$ with $1\leq m\in\mathbb{Z}$.
The morphisms in $t\Delta$ are generated by the following morphisms
\[d^i:[m-1]\to [m], 0\leq i\leq m,\]
\[s^i:[m+1]\to [m], 0\leq i\leq m,\]
\[\varphi:[m]\to [m]_t,\] 
\[\zeta^i:[m+1]_t\to [m], 0\leq i\leq m,\]
subject to the usual cosimplicial identities on $d^i$'s and $s^i$'s, and the following additional relations
\begin{itemize}
\item $\zeta^i\varphi=s^i:[m+1]\to[m]$, $1\leq m$ and $0\leq i\leq m$
\item $s^i\zeta^{j+1}=s^j\zeta^i:[m+2]_t\to[m]$, $0\leq i\leq j\leq m$
\end{itemize}
We may view $\Delta$ as a subcategory of $t\Delta$ in the evident way.

\begin{definition}[\cite{OR}]We let $\mathbf{Set}$ denote the category of sets and maps.
A prestratified simplicial set $X$ is a functor $X:t\Delta^{op}\to\mathbf{Set}$.
A stratified simplicial set $X$ is a prestratified simplicial set such that the maps
\[\varphi^*:X([m]_t)\to X([m])\]
are injective for all $m\geq 1$.
We let $\mathbf{Set}^{t\Delta^{op}}$ denote the category of prestratified simplicial sets and natural transformations and let $\mathbf{msSet}$ denote the full subcategory of stratified simplicial sets.
We call an element in $X([n]_t)$ a marked $n$-simplex for any $X\in\mathbf{Set}^{t\Delta^{op}}$ and $n>0$.
\end{definition}

\begin{definition}[\cite{OR}]Let $n$ be a natural number and $k\in[n]$.
\begin{itemize}

  \item {\it The standard thin $n$-simplex} $\Delta[n]_t$ is the simplicial set with marking whose underlying simplicial set is the standard simplicial set $\Delta[n]$ and  
  \[
  m\Delta[n]_t = \begin{cases}
    d\Delta[n]\cup\{\operatorname{Id}_{[n]}\} & (n\neq 0) \\
    d\Delta[n] & (n=0).
  \end{cases}
\]
  
  \item {\it The $k$-admissible $n$-simplex} $\Delta^k[n]$ with $n\geq 1$\footnote{In the case $n=0$, hence $k=0$, we define $\Delta^0[0]$ to be the simplicial set with marking whose underlying simplicial set is $\Delta[0]$ and $m\Delta^0[0]=d\Delta[0]$.} is the simplicial set with marking whose underlying simplicial set is the standard simplicial set $\Delta[n]$ and
  \[m\Delta^k[n]=d\Delta[n]\cup\{\alpha\in\Delta[n]| \{k-1, k, k+1\}\cap[n]\subset\operatorname{Im}(\alpha)\}.\] 
  
  \item {\it The $(n-1)$-dimensional $k$-admissible horn} $\Lambda^k[n]$ with $n\geq 1$ is the regular simplicial subset with marking of $\Delta^k[n]$ whose underlying simplicial set is the usual simplicial $k$-th horn.
  
  \item $\Delta^k[n]''$ (respectively, $\Lambda^k[n]'$) is the simplicial set with marking whose underlying simplicial set is the same as that of $\Delta^k[n]$ (respectively, $\Lambda^k[n]$) and its marked simplices are $m\Delta^k[n]$ (respectively, $m\Lambda^k[n]$) with all its $(n-1)$-simplices.
  
  \item $\Delta^k[n]':=\Delta^k[n]\cup\Lambda^k[n]'.$
  \item $\Delta[3]_{eq}$ is the simplicial set with marking whose underlying simplicial set is $\Delta[3]$ and all $n$-simplices for $n\geq2$ and the non-degenerate two 1-simplices  $\overline{02}$ and $\overline{13}$ are marked, where
 \[\overline{02}:[1]\to [3], \ \ \overline{02}(0)=0, \  \overline{02}(1)=2,\] 
  \[\overline{13}:[1]\to [3], \ \ \overline{13}(0)=1, \  \overline{13}(1)=3.\] 
   \item $\Delta[3]^{\sharp}$ is the simplicial set with marking whose underlying simplicial set is $\Delta[3]$ with \[m\Delta[3]^{\sharp}=\bigcup_{n\geq1}\Delta[3]_n.\]
  \end{itemize}
\end{definition}

\begin{rem}(1)Note that the functor $K$ preserves admissible simplices.
More precisely, for any morphism $\alpha:[r]\to[n]$ with $\{k-1, k, k+1\}\cap[n]\subset\operatorname{Im}(\alpha)$, the morphism $K(\alpha):[r+1]\to[n+1]$ satisfies that $\{k-1, k, k+1\}\cap[n+1]\subset\operatorname{Im}(K(\alpha))$.
The same holds for $J$.

(2)There are evident inclusions $\Lambda^k[n]\to\Delta^k[n]$, $\Delta^k[n]'\to\Delta^k[n]''$, $\Delta[3]_{eq}\to\Delta[3]^{\sharp}$, and $\Delta[n]\star\Delta[3]_{eq}\to\Delta[n]\star\Delta[3]^{\sharp}$ for all $k, n$.
We may call them elementary anodyne extensions.
Note that by definition the inclusions $\Delta[0]\to\Delta[1]_t$ are the elementary anodyne extensions $\Lambda^k[1]\to\Delta^k[1]$ with $k\in\{0, 1\}$.

We call prestratified simplicial sets having the right lifting property with respect to elementary anodyne extensions precomplicial sets (\cite{OR}).
\end{rem}

The following model structures are expected to model $(\infty, \infty)$-categories.
\begin{thm}[\cite{OR}]The category $\mathbf{Set}^{t\Delta^{op}}$ of prestratified simplicial sets admits a model structure in which
\begin{itemize}
\item the cofibrations are precisely the monomorphisms
\item the fibrant objects are precisely the precomplicial sets
\end{itemize}
These classes of morphisms give rise to a cofibrantly generated model structure on the category $\mathbf{msSet}$ of stratified simplicial sets.
These two model structures are Quillen equivalent.
\end{thm}
We may call these model structures Ozornova-Rovelli model structures.

The following simple morphisms in $\Delta$ (or in $t\Delta$) are used to define the lax Gray-Verity product.

\begin{definition}[\cite{V1}]For any $(p, q)\in\mathbb{N}^2$, there are four maps in $\Delta$:

\[\modelsv^{p, q}_1:[p]\to [p+q], \ \modelsv^{p, q}_1(i)=i,\]

\[\modelsv^{p, q}_2:[q]\to [p+q], \ \modelsv^{p, q}_2(i)=i+p,\]

\[\vmodels^{p, q}_1:[p+q]\to[p], \ \vmodels^{p, q}_1(i) =
\begin{cases}
i & 0 \leq i\leq p,\\
p & p<i\leq p+q,
\end{cases}\]

\[\vmodels^{p, q}_2:[p+q]\to[q], \ \vmodels^{p, q}_1(i) =
\begin{cases}
0 & 0 \leq i< p,\\
i-p & p\leq i\leq p+q,
\end{cases}\]

\end{definition}

\begin{rem}\label{tsunagi}
Let $p, q\in\mathbb{N}$.
We note that the morphisms $\modelsv_{1}^{p, q}:[p]\to[p+q]$ (resp. $\modelsv_{2}^{p, q}:[q]\to[p+q]$) in $\Delta$ is compatible with $J:\Delta\to\Delta$ (resp. $K:\Delta\to\Delta$).
More precisely, 
\[J(\modelsv_{1}^{p, q})=\modelsv_{1}^{p+1, q},  K(\modelsv_{2}^{p, q})=\modelsv_{2}^{p, q+1}\]
hold.
Similarly, we also have 
\[J(\vmodels_{1}^{p, q})=\vmodels_{1}^{p+1, q}, K(\vmodels_{2}^{p, q})=\modelsv_{2}^{p, q+1}.\]

In addition, the morphisms $\modelsv_{1}^{p, q}$ and $\modelsv_{2}^{p, q}$ in $\Delta$ correspond each other via the endofunctor $\operatorname{rev}:\Delta\to\Delta$.
Since by definition $\modelsv_{1}^{p, q}=s^p\circ s^{p+1}\circ\cdots\circ s^{p+q-1}$ and $\modelsv_{2}^{q, p}=s^0\circ s^{0}\circ\cdots\circ s^{0}$ hold, so by the definition of $\operatorname{rev}$ we have 
\[\operatorname{rev}(\modelsv_{1}^{p, q})=\modelsv_{2}^{q, p}, \ \operatorname{rev}(\modelsv_{2}^{p, q})=\modelsv_{1}^{q, p}.\]
Also, we have the equations below
\[\vmodels^{p, q}_1\circ\modelsv^{p, q}_1=\operatorname{id}, \ \ \vmodels^{p, q}_2\circ\modelsv^{p, q}_2=\operatorname{id}\]
\end{rem}

\begin{rem}\label{shuffle} As is demonstrated in \cite{V0}, the non-degenerate $(p+q)$-simplices in the simplicial set $\Delta[p]\times\Delta[q]$ correspond to the shortest paths from the left bottom corner $(0,0)$ to the right upper corner $(p,q)$ in the figure below.

 
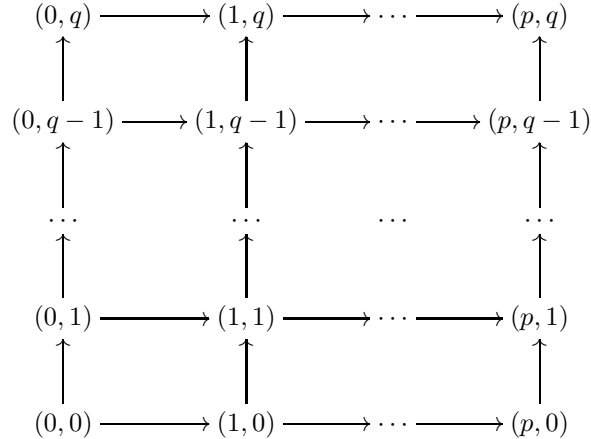
\[
   \xymatrix{
 (0, q) \ar[r]& (1, q)\ar[r]& \cdots\ar[r]&(p, q)\\
 (0, q-1)\ar[r]\ar[u]& (1, q-1)\ar[r]\ar[u]& \cdots\ar[r]& (p, q-1)\ar[u] \\ 
 \cdots\ar[u]& \cdots\ar[u]& \cdots& \cdots\ar[u]& \\     
 (0,1)\ar[r]\ar[u]& (1,1)\ar[r]\ar[u]& \cdots\ar[r]& (p,1)\ar[u]\\ 
 (0, 0) \ar[r]\ar[u]& (1, 0)\ar[r]\ar[u]& \cdots\ar[r]&\ar[u] (p, 0)\\ 
}
\]

For instance, the $(p+q)$-simplex $(\vmodels^{p, q}_1, \vmodels^{p, q}_2)\in(\Delta[p]\times\Delta[q])([p+q])$ corresponds to the shortest path turning the right bottom corner $(p,0)$.
\end{rem}

\begin{definition}[\cite{V1}]Let $X, Y\in\mathbf{msSet}$. Their lax Gray-Verity product $X\otimes Y$ is the following stratified simplicial set:
the underlying simplicial set is the cartesian product $X\times Y$ and an $n$-simplex $x\otimes y\in X \otimes Y$ is marked if and only if for any $(p, q)\in\mathbb{N}^2$ with $p+q=n$, $x\cdot\modelsv^{p, q}_1\in X([p]_t)$ or $y\cdot\modelsv^{p, q}_2\in Y([q]_t)$.
\end{definition}

\begin{rem}
Let $\Delta[l]$ denote the prestratified simplicial set represented by $[l]\in t\Delta$, which is a stratified simplicial set.
Note that there is only one unmarked $(m+n)$-simplex $(\vmodels^{m, n}_1, \vmodels^{m, n}_2)$ in $\Delta[m]\otimes\Delta[n]$. 
Any other $(m+n)$-simplex is marked by definiton.
Intuitively speaking, discarding marked simplices, $\Delta[m]\otimes\Delta[n]$ is similar to $\Delta[m+n]$, while the join of $\Delta[m]$ and $\Delta[n]$ is $\Delta[m+n+1]$.
Note that $\Delta[m]\otimes\Delta[n]$ has one or more unmarked $k$-simplices with $k\leq m+n$ when $m, n\geq 1$.
\end{rem}

We will extend the join construction of simplicial sets recalled in the last section to prestratified simplicial sets.
To do so, abusing the notation, we now extend the concatenation functor $\star:\Delta\times\Delta\to\Delta$ to $\star:t\Delta\times t\Delta\to t\Delta$.
Roughly speaking, viewing $\Delta\subset t\Delta$ and $\star|_{\Delta\times\Delta}=\star$, $\zeta^i$'s and $\varphi$'s behave like $s^i$'s and the identity morphisms respectively.

\begin{definition}
We define the concatenation functor $\star:t\Delta\times t\Delta\to t\Delta$ as follows.
\begin{enumerate}
\item On the objects in $t\Delta$, $\star$ acts as follows:

For $m, n\geq 1$,
\[[m]_t\star[n]=[m]\star[n]_t=[m]_t\star[n]_t=[m+n+1]_t,\] 
\[[m]\star[n]=[m+n+1].\]

For $m\geq 0$,
\[[0]\star[m]=[m]\star[0]=[m+1].\]

For $m\geq 1$,
\[[0]\star[m]_t=[m]_t\star[0]=[m+1]_t.\]

\item On the generators of morphisms in $t\Delta$, $\star$ acts as follows:

Viewing $d^i$'s and $s^j$'s in $t\Delta$ as in $\Delta$, their concatenations are defined as in Section 2.1.

For $\zeta^j:[m+1]_t\to[m]$ and $d^i:[n-1]\to[n]$,
\[d^i\star \zeta^j=d^i\circ \zeta^{n+j}, \ \zeta^j\star d^i=d^{m+1+i}\circ \zeta^j.\]
For $\zeta^j:[m+1]\to[m]_t$ and $s^i:[n+1]\to[n]$,
\[s^i\star \zeta^j=s^{i}\circ \zeta^{n+1+j}, \ \zeta^j\star s^i=s^{m+1+i}\circ \zeta^j.\]
For $\zeta^j:[m+1]\to[m]_t$ and $\zeta^i:[n+1]\to[n]_t$,
\[\zeta^i\star \zeta^j=\zeta^{i}\circ \zeta^{n+1+j}, \ \zeta^j\star s^i=s^{m+1+i}\circ \zeta^j.\]
For $\varphi:[n]\to[n]_t$ and $d^i:[m-1]\to[m]$,
\[\varphi\star d^i=\varphi\circ d^{n+1+i}, \ d^i\star\varphi=\varphi\circ d^i.\]
For $\varphi:[n]\to[n]_t$ and $s^i:[m+1]\to[m]$,
\[\varphi\star s^i=\varphi\circ s^{n+1+i}, \ s^i\star\varphi=\varphi\circ s^i.\]
For $\varphi:[n]\to[n]_t$ and $\zeta^i:[m+1]_t\to[m]$,
\[\varphi\star \zeta^i=\varphi\circ \zeta^{n+1+i}, \ \zeta^i\star\varphi=\varphi\circ \zeta^i.\]
For $\varphi_n:[n]\to[n]_t$ and $\varphi_m:[m]\to[m]_t$, where we put subscripts for convenience, 
\[\varphi_n\star \varphi_m=\varphi_m\star\varphi_n=\varphi_{m+n}.\]
\end{enumerate}

\end{definition}

\begin{rem}One might define $[m]\star[n]_t$ to be $[m+n+1]$ for $m, n\geq 1$, namely the concatenation of an unmarked simplex and a marked simplex should be unmarked.
But it  and $[0]\star[n]_t=[n+1]_t$ would imply $[1]\star[n]_t=([0]\star[0])\star[n]_t\neq[0]\star([0]\star[n]_t)=[0]\star[n+1]_t$.
\end{rem}

We define the join functor $(\mathchar`-)\oplus(\mathchar`-):\mathbf{Set}^{t\Delta^{op}}\times\mathbf{Set}^{t\Delta^{op}}\to\mathbf{Set}^{t\Delta^{op}}$ to be the Day convolution of $\star:t\Delta\times t\Delta\to t\Delta$.
Note that our definition is compatible with the join of stratified simplicial sets \cite[Definition 33]{V1}.

\begin{prop}For any $X\in\mathbf{msSet}$, the functors $(\mathchar`-)\oplus X, X\oplus(\mathchar`-):\mathbf{msSet}\to\mathbf{msSet}$ are left Quillen functors with respect to Ozornova-Rovelli model structure.
\end{prop}

\begin{proof}In \cite[Chapter 6]{V1}, it has been proven that for any stratified simplicial set $X$, $(\mathchar`-)\oplus X$ and $X\oplus(\mathchar`-)$ are left Quillen functors with respect to the model structure for non-saturated weak complicial sets on $\mathbf{msSet}$.

It is enough to show that the morphism $\Delta[3]_{eq}\oplus X\to\Delta[3]^\sharp\oplus X$ is a trivial cofibration with respect to Ozornova-Rovelli model structure.
Since $\Delta[3]_{eq}\oplus \Delta[n]_?\to\Delta[3]^\sharp\oplus \Delta[n]_?$ is a trivial cofibration for any $n\in\mathbb{N}$, where $[n]_?$ denotes $[n]$ or $[n]_t$, and the join construction is compatible with colimits, the morphism $\Delta[3]_{eq}\oplus X\to\Delta[3]^\sharp\oplus X$ is a trivial cofibration.
\end{proof}

Since every stratified simplicial set is cofibrant, we obtain the following.
\begin{cor}\label{1}Let $X, X', Y, Y'\in\mathbf{msSet}$. 
Suppose we have weak equivalences $f:X\to X'$ and $g:Y\to Y'$.
Then the morphism $f\oplus g:X\oplus Y\to X'\oplus Y'$ is also a weak equivalence.
\end{cor}

\begin{rem}In \cite[\S 1.2.8]{Lur}, it is proven that the join of two quasi-categories is a quasi-category. 
The same argument there shows that the join of two precomplicial sets has the right lifting property with respect to the inclusions $\Lambda^k[n]\to\Delta^k[n]$, but it may not show that the join has the right lifting property with respect to the other anodyne extensions.
\end{rem}

\section{Results}

\subsection{Stratified stabilization of simplicial sets}
In this section, by using the functor $\star:t\Delta\times t\Delta\to t\Delta$, we define stratified analogue of $\Delta_{st}$, $\Delta_{st'}$ and $\Delta_{st^2}$.
To do that, by abusing notation, we define the following shift functors.

\begin{definition}
We denote by $K:t\Delta\to t\Delta$ the shift functor $(\mathchar`-)\star[0]:t\Delta\to t\Delta$ and by $J:t\Delta\to t\Delta$ the other shift functor $[0]\star(\mathchar`-):t\Delta\to t\Delta$
\end{definition}

More explicitly, these functors act on the generators of morphisms as follows:

\[K(d^i:[n-1]\to[n])=d^i:[n]\to[n+1],\]
\[K(s^i:[n+1]\to[n])=s^i:[n+2]\to[n+1],\]
\[K(\varphi:[n]\to[n]_t)=\varphi:[n+1]\to[n+1]_t,\] 
\[K(\zeta^i:[n+1]_t\to[n])=\zeta^i:[n+2]\to[n+1]_t,\]
\[J(d^i:[n-1]\to[n])=d^{i+1}:[n]\to[n+1],\] 
\[J(s^i:[n+1]\to[n])=s^{i+1}:[n+2]\to[n+1],\] 
\[J(\varphi:[n]\to[n]_t)=\varphi:[n+1]\to[n+1]_t,\] 
\[J(\zeta^i:[n+1]_t\to[n])=\zeta^{i+1}:[n+2]\to[n+1]_t.\]

By using these functors we obtain the following categories:
\[t\Delta_{st}:=\operatorname{colim}(t\Delta\xrightarrow{K}t\Delta\xrightarrow{K}t\Delta\xrightarrow{K}\cdots)\]
\[t\Delta_{st'}:=\operatorname{colim}(t\Delta\xrightarrow{J}t\Delta\xrightarrow{J}t\Delta\xrightarrow{J}\cdots)\]
\[t\Delta_{st^2}:=\operatorname{colim}(t\Delta\xrightarrow{K}t\Delta\xrightarrow{J}t\Delta\xrightarrow{K}t\Delta\xrightarrow{J}t\Delta\xrightarrow{K}\cdots)\]

\begin{rem}(1)By definition, we have $J\circ K=K\circ J$.
So we can change the order of the functors in the colimit diagram of $t\Delta_{st^2}$.

(2)These three categories admit the following descriptions:

By abusing notation, the objects of $t\Delta_{st}$ may be denoted by $[n]$ or $[n]_t$ for all $n\in\mathbb{Z}$.
The morphism of $t\Delta_{st}$ are generated by the morphisms
\[d^i:[n-1]\to[n], \ s^j:[n+1]\to[n], \ \zeta^k:[n+1]_t\to[n], \ \varphi:[n]\to[n]_t\]
for all $n\in\mathbb{Z}$ and integers $i, j, k\geq 0$ subject to the stratified version of the identities in Remark \ref{identities}.
Note that there is $[0]_t\in t\Delta_{st}$

Dually, the objects of $t\Delta_{st'}$ are denoted by $[n]$ or $[n]_t$ for all $n\in\mathbb{Z}$.
The morphism of $t\Delta_{st'}$ are generated by the morphisms
\[d^i:[n-1]\to[n], \ s^j:[n+1]\to[n], \ \zeta^k:[n+1]_t\to[n], \ \varphi:[n]\to[n]_t\]
for all $n\in\mathbb{Z}$ and integers $i, j, k\leq n$ subject to the same identities.

The objects of $t\Delta_{st^2}$ are denoted by $[n]$ or $[n]_t$ for all $n\in\mathbb{Z}$.
The morphism of $t\Delta_{st^2}$ are generated by the morphisms
\[d^i:[n-1]\to[n], \ s^j:[n+1]\to[n], \ \zeta^k:[n+1]_t\to[n], \ \varphi:[n]\to[n]_t\]
for all $n, i, j\in\mathbb{Z}$ subject to the same identities.

We may view in the evident way the categories $\Delta_{st}$, $\Delta_{st'}$, and $\Delta_{st^2}$ as subcategories of $t\Delta_{st}$, $t\Delta_{st'}$, and $t\Delta_{st^2}$ respectively.

(3)Abusing notation, we define the stratified analogue of $\operatorname{rev}:\Delta_{st}\to\Delta_{st'}$.
To do that, we define $\operatorname{rev}:t\Delta\to t\Delta$  as follows:

On the objects, $\operatorname{rev}([n])=[n]$ and $\operatorname{rev}([n]_t)=[n]_t$.
On the generators of morphisms,
\[\operatorname{rev}(d^i:[n-1]\to[n])=d^{n-i}:[n-1]\to[n],\] \[\operatorname{rev}(s^i:[n+1]\to[n])=s^{n-i}:[n+1]\to[n],\]\[\operatorname{rev}(\varphi:[n]\to[n]_t)=\varphi:[n]\to[n]_t,\] \[\operatorname{rev}(\zeta^i:[n+1]_t\to[n])=\zeta^{n-i}:[n+1]\to[n]_t.\]

This fits into the commutative diagram
 \[
   \xymatrix{
 t\Delta\ar^{K}[r] \ar_{\operatorname{rev}}[d]& t\Delta\ar^{\operatorname{rev}}[d]\\
 t\Delta\ar_{J}[r]&t\Delta.\\     
}
\]
This defines a functor $t\Delta_{st}\to t\Delta_{st'}$, which we also denote by $\operatorname{rev}$ abusing notation.

\end{rem}

By using the functor $\star:t\Delta\times t\Delta\to t\Delta$, we extend the functor $\star:\Delta_{st'}\times\Delta_{st}\to \Delta_{st^2}$ to $\star:t\Delta_{st'}\times t\Delta_{st}\to t\Delta_{st^2}$, again abusing notation.

Let $\theta:[n]_?\to[n']_?\in t\Delta_{st'}$ and $\tau:[m]_?\to[m']_?\in t\Delta_{st}$, where $[n]_?$ denotes $[n]$ or $[n]_t$.
By the definitions of $t\Delta_{st'}$ and $t\Delta_{st}$, there exist $k, l\in\mathbb{N}$ and morphisms $\tilde{\theta}, \tilde{\tau}$ in $t\Delta$ such that
\[\tilde{\theta}:[n+k]_?\to[n'+k]_?\in t\Delta\]
represents $\theta$ and
\[\tilde{\tau}:[m+l]_?\to[m'+l]_?\in t\Delta\]
represents $\tau$.
Denote $\theta\star\tau:[n]_?\star[m]_?\to[n']_?\star[m']_?$ the map in $\Delta_{st^2}$ represented by
\[\tilde{\theta}\star\tilde{\tau}:[n+m+k+l]_?\to[n'+m'+k+l]_?\in\Delta.\]
This construction defines a functor $t\Delta_{st'}\times t\Delta_{st}\to t\Delta_{st^2}$.

\begin{definition}We let $\mathbf{Set}_*^{t\Delta_{st}^{op}}$ denote the category of functors $t\Delta_{st}^{op}\to \mathbf{Set}_*$ and call the objects prestratified stable simplicial sets.
For any $X\in\mathbf{Set}_*^{t\Delta_{st}^{op}}$ and $[n]_t\in t\Delta_{st}$, we call elements in $X([n]_t)$ marked $n$-simplices.
\end{definition}

\begin{rem}For any $A\in\mathbf{Set}^{t\Delta^{op}}$, we may write $\Sigma_{+}^{\infty}A$ for the corresponding prestratified stable simplicial set.

\[\Sigma_{+}^{\infty}A([n])=
\begin{cases}
A([n])\coprod\{*\} & n\geq 0,\\
\{*\} & n<0,
\end{cases}\]
and $d_j\alpha=*$ for any $\alpha\in\Sigma_{+}^{\infty}A([m])$ and $j>m$.
\end{rem}

As we have defined the join product on $\mathbf{Set}_*^{\Delta_{st}^{op}}$, we here define an analogous product on $\mathbf{Set}_*^{t\Delta_{st}^{op}}$.
For $X, Y\in\mathbf{Set}_*^{t\Delta_{st}^{op}}$, we let $X\wedge'' Y$ denote the point-wise smash product, namely $(X\wedge'' Y)([n])=X([n])\wedge Y([n])$, where $\wedge$ denotes the smash product of pointed sets.

Let $X\wedge' Y$ denote the left Kan extension of $X\wedge''Y$ with respect to 
\[t\Delta_{st}^{op}\times t\Delta_{st}^{op}\xrightarrow{\operatorname{rev}\times \operatorname{id}} t\Delta_{st'}^{op}\times t\Delta_{st}^{op}\xrightarrow{\star} t\Delta_{st^2}^{op}.\]

This construction defines a functor $(\mathchar`-)\wedge'(\mathchar`-):\mathbf{Set}_*^{\Delta_{st}^{op}}\times \mathbf{Set}_*^{\Delta_{st}^{op}}\to\mathbf{Set}_*^{\Delta_{st^2}^{op}}$, which is the staratified analogue of the join construction of stable simplicial sets.
To define a product on $\mathbf{Set}_*^{t\Delta_{st}^{op}}$, the stratified analogue of the Kan-Whitehead product on stable simplicial sets, we introduce a functor $\mathbf{Set}_*^{t\Delta_{st^2}^{op}}\to\mathbf{Set}_*^{t\Delta_{st}^{op}}$.

\begin{definition}Let $V\in\mathbf{Set}_*^{t\Delta_{st^2}^{op}}$.
We define $V_{-1}\in\mathbf{Set}_*^{t\Delta_{st}^{op}}$ as follows.
For any $[n], [m]_t\in t\Delta_{st}^{op}$, we set
\[V_{-1}([n]):=\{x\in V([n+1])|d^V_jx=*, j<1\},\]
\[V_{-1}([m]_t):=\{x\in V([n+1]_t)|d^V_j\varphi^V x=*, j<1\}.\]
The operators are given as follows:
\[d_i:=d_{i+1}^V, \ s_i:=s_{i+1}^V, \ \varphi:=\varphi^V, \ \zeta_i:=\zeta_{i+1}^V,\]
where $d_{j}^V$, $s_{j}^V$, $\varphi^V$ and $\zeta_{j}^V$ are the operators for $V$.
\end{definition}

\begin{definition}Let $X, Y\in\mathbf{Set}_*^{t\Delta_{st}^{op}}$. The stratified Kan-Whitehead smash product $X\wedge Y$ is $(X\wedge'Y)_{-1}$.
\end{definition}

By Remark \ref{tsunagi}, the following is well defined.

\begin{definition}Let $p, q\in\mathbb{Z}$.

(1) We denote by $\modelsv_{1}^{p, q}:[p]\to[p+q]$ the morphism in $\Delta_{st'}$ represented by a morphism $\modelsv_{1}^{p+k, q}:[p+k]\to[p+q+k]$ in $\Delta$ for some $k\in\mathbb{N}$.
Similarly, we denote by $\modelsv_{2}^{p, q}:[q]\to[p+q]$ the morphism in $\Delta_{st}$ represented by a morphism $\modelsv_{2}^{p, q+k}:[q+k]\to[p+q+k]$ in $\Delta$ for some $k\in\mathbb{N}$.

(2) We denote the morphism $\operatorname{rev}(\modelsv_{1}^{p, q}):[p]\to[p+q]$ in $\Delta_{st'}$ by $\modelsv_{2}^{q, p}:[p]\to[p+q]$.
Similarly, we denote the morphism $\operatorname{rev}(\modelsv_{2}^{p, q}):[q]\to[p+q]$ in $\Delta_{st}$ by $\modelsv_{1}^{q, p}:[q]\to[p+q]$.
\end{definition}

We may view these morphisms as those in $t\Delta_{st}$ and $t\Delta_{st'}$ respectively.

\begin{definition}
Let $X, Y\in \mathbf{Set}_*^{t\Delta_{st}^{op}}$.
The stable analogue of lax Gray-Verity product $X\tilde{\otimes}Y$ is defined as follows:
\begin{itemize}
\item its underlying stable simplicial set is the point-wise smash product $X\wedge Y$
\item an $n$-simplex $x\tilde{\otimes}y\in X\tilde{\otimes}Y$ is marked if and only if for any $(p, q)\in\mathbb{Z}^2$ with $p+q=n$, $x\circ\modelsv^{p, q}_1\in X([p]_t)$ or $y\circ\modelsv^{p, q}_2\in Y([q]_t)$.
\end{itemize}

\end{definition}

This product $\tilde{\otimes}$ is a straightforward analogue of lax Gray tensor product for stratified simplicial sets.
The lax Gray-Verity product plays a pivotal role in weak $\omega$-category theory (\cite{V2} and \cite{ORV}).
Thus, it would be fair to expect that its stable analogue also would pray a role in stable objects.
We below construct a natural morphism between it and stratified Kan-Whitehead product.

\begin{prop}
There exists a natural morphism $X\tilde{\wedge}Y\to X\tilde{\otimes}Y$.
\end{prop}
\begin{proof}Let $X, Y\in\mathbf{Set}_*^{t\Delta_{st}^{op}}$ and $x\tilde{\wedge}y$ be an unmarked $n$-simplex of $X\tilde{\wedge}Y$.
Then by the definition of $X\tilde{\wedge}Y$ there exit $p, q\in\mathbb{Z}$ with $p+q=n$ such that $x\in X([p])$ and $y\in Y([q])$.
We assign to it the $n$-simplex $(x\circ\vmodels^{p, q}_1, y\circ\vmodels^{p, q}_2)$ of $X\tilde{\otimes}Y$.

Let $x\tilde{\wedge}y$ be a marked $n$-simplex of $X\tilde{\wedge}Y$.
Then there exit $p, q\in\mathbb{Z}$ with $p+q=n$ such that 
\begin{enumerate}
\item $x\in X([p]_t)$ and $y\in Y([q])$ or
\item $x\in X([p])$ and $y\in Y([q]_t)$ or
\item $x\in X([p]_t)$ and $y\in Y([q]_t)$.
\end{enumerate}

For each case, we again consider $(x\circ\vmodels^{p, q}_1, y\circ\vmodels^{p, q}_2)$ of $X\tilde{\otimes}Y$.
We need to show that $(x\circ\vmodels^{p, q}_1, y\circ\vmodels^{p, q}_2)$ is marked.
It is enough to see the first case, the others are the same.
Since $(x\circ\vmodels^{p, q}_1)\circ\modelsv^{p, q}_1=x$ and $x\in X$ is marked, $(x\circ\vmodels^{p, q}_1, y\circ\vmodels^{p, q}_2)\in X\tilde{\otimes}Y$ is also marked by definition. 
This defines a natural map $X\tilde{\wedge}Y\to X\tilde{\otimes}Y.$
\end{proof}

\begin{rem}As we have seen in Remark \ref{shuffle}, any non-degenerate $(p+q)$-simplex in $\Delta[p]\times\Delta[q]$ can be expressed as a shortest path in the ordered set $[p]\times[q]$.
The simplex $(\vmodels^{p, q}_1, \vmodels^{p, q}_2)$ corresponds to the most right-lower path.

We could take another pair of surjections $\theta:[p+q]\to[p]$ and $\tau:[p+q]\to[q]$ and consider the map $X([p])\times Y([q])\to (X\tilde{\otimes}Y)_{p+q}$, $(x, y)\mapsto (x\circ\theta, y\circ\tau)$.
But by definition, the simplices $x\circ\theta\circ\modelsv^{p, q}_1$ and $y\circ\tau\circ\modelsv^{p, q}_2$ are degenerate, hence marked.
Thus, $(x\circ\theta, y\circ\tau)\in X\tilde{\otimes}Y$ is always marked.

In this sense, the map we constructed in the proof above may be the only suitable one.
\end{rem}

\subsection{A homotopy theory for stable precomplicial sets}
In this section, we introduce a homotopy theory on prestratified stable simplicial sets, which is a straight forward analogue of the structure of a category of fibrant objects on combinatorial spectra given by Brown \cite{B}.

Mimicking Kan's suspension functor, we introduce the following functor.
\begin{definition}
The suspension functor $\operatorname{S}:\mathbf{Set}^{t\Delta^{op}}\to\mathbf{Set}_*^{t\Delta^{op}}$ is the left Kan extension along yoneda embedding of
\[t\Delta\to \mathbf{Set}_*^{t\Delta^{op}}, [n]_?\mapsto (\Delta[n+1]_?)_+/(\Delta[n]_?)_+\wedge \Delta[0]_+),\]
where $(\mathchar`-)_+$ denotes the functor $\mathbf{Set}^{t\Delta^{op}}\to\mathbf{Set}_*^{t\Delta^{op}}$ adding the base points.
\end{definition}

\begin{definition}
(1) A prestratified prespectrum $L$ consists of

(i) a sequence of pointed prestratified simplicial sets $L_i$ with $i\in\mathbb{N}$,

(ii) a sequence of monomorphisms $\lambda_i:\operatorname{S}L_i\to L_{i+1}$ of pointed prestratified simplicial sets with $i\in\mathbb{N}$.

(2) A morphism $\psi:\{L_i, \lambda_i\}\to \{M_i, \mu_i\}$ of prestratified prespectra is a sequence of morphisms $\psi_i:L_i\to M_i$ of pointed prestratified simplicial sets such that $\psi_{i+1}\circ\lambda_i=\mu_i\circ\operatorname{S}\psi_i$.

(3) A morphism $\psi:\{L_i, \lambda_i\}\to \{M_i, \mu_i\}$ of prestratified prespectra is called a weak equivalence if for every $i\in\mathbb{N}$, $\psi_i$ is a weak equivalence for the pointed Ozornova-Rovelli model structure. 
\end{definition}

\begin{definition}Let $X\in\mathbf{Set}_*^{t\Delta_{st}^{op}}$.
We define the corresponding prestratified prespectrum $\operatorname{Ps}(X)=\{X_i, \xi_i\}$ as follows.

For any $i\in\mathbb{N}$, the pointed set of $n$-simplices in $X_i$ is given by
\[X_i([n]):=\{\alpha\in X([n-i])|d_0\cdots d_n\alpha=*, d_j\alpha=*,   (j\geq n)\},\]
where $*$ denotes (the degeneracy of) the base point.
Similarly, the pointed set of marked $n$-simplices in $X_i$ is given by
\[X_i([n]_t):=\{\alpha\in X([n-i]_t)|d_0\cdots d_n\varphi^*\alpha=*, d_j\varphi^*\alpha=*,   (j\geq n)\},\]
The structure morohisms on $X_i$ will be induced by those of $X$ and $X_i$ is indeed a pointed prestratified simplicial set with them.

For any $i\in\mathbb{N}$, the monomorphism $\xi_i:\operatorname{S}X_i\to X_{i+1}$ is given by $\xi_i(\alpha, \phi)=\alpha$.
\end{definition}

\begin{rem}This defines a functor $\operatorname{Ps}:\mathbf{Set}_*^{t\Delta_{st}^{op}}\to\mathbf{ppSp}$, by letting $\mathbf{ppSp}$ denote the category of prestratified prespectra..
By construction, for any $i\in\mathbb{N}$ $X_i$ is a stratified simplicial set if $X$ is a stratified stable simplicial set.
\end{rem}

Brown has introduced the following notion in the study of shaves valued in simplicial sets.
\begin{definition}[\cite{B}]
Let $\mathcal{C}$ be a category with finite products and a final object denoted by $*$.
Assume that $\mathcal{C}$ has two distinguished classes of morphisms, called the weak equivalences and the fibrations.
A morphism in $\mathcal{C}$ will be called an aspherical fibration if it is both a weak equivalence and a fibration.
We define a path object for an object $B$ in $\mathcal{C}$ to be an object $B^I$ together with morphisms
\[B\xrightarrow{s} B^I\xrightarrow{(d_0, d_1)} B\times B,\]
where $s$ is a weak equivalence, $(d_0, d_1)$ is a fibration, and the composite is the diagonal morphism.

We call $\mathcal{C}$ a category of fibrant objects if the following are satisfied

(A) Let $f$ and $g$ be morphisms such that $gf$ is defined.
If two of $f$, $g$, $gf$ are weak equivalences then so is the third.
Any isomorphism is a weak equivalence.

(B) The composite of two fibrations is a fibration.
Any isomorphism is a fibration.

(C) For any morphism $A\to C$ and a fibration (resp. aspherical fibration) $B\to C$, the pullback $A\times_{C}B$ exists and the projection $A\times_{C}B\to A$ is a fibration (resp. aspherical fibration).

(D) For any object, there exists at least one path object.

(E) For any object $B$ the morphism $B\to*$ is a fibration.

\end{definition}

We define weak equivalences and fibrations in $\mathbf{Set}_*^{t\Delta_{st}^{op}}$.

\begin{definition}\label{wf}
A morphism $f$ in $\mathbf{Set}_*^{t\Delta_{st}^{op}}$ is 
\begin{enumerate}
\item a weak equivalence if for any $i\in\mathbb{N}$ the morphism $\operatorname{Ps}(f)_i$ of pointed prestratified simplicial sets is a weak equivalence for Ozornova-Rovelli model structure,
\item a fibration if for any $i\in\mathbb{N}$ the morphism $\operatorname{Ps}(f)_i$ of pointed prestratified simplicial sets is a fibration for Ozornova-Rovelli model structure.
\end{enumerate}
We call a prestratified stable simplicial set $X$ is fibrant if the unique morphism $X\to *$ is a fibration.
\end{definition}

\begin{prop}For any fibrant stratified stable simplicial set $B$,  there exists the diagram of fibrant objects
\[B\xrightarrow{s}\operatorname{Path}(B)\xrightarrow{(d_0, d_1)}B\times B,\]
where $s$ is a weak equivalence, $(d_0, d_1)$ is a fibration, and the composite is the diagonal morphism.
\end{prop}

\begin{proof}
Let $B$ be a fibrant stratified stable simplicial set.
Then we have a prestratified prespectrum $\operatorname{Ps}(B)=\{B_i, \beta_i\}$, where $B_i$ is a pointed stratified simplicial set for any $i\in\mathbb{N}$.
For any $i\in\mathbb{N}$, we have the mapping stratified simplicial set $(B_i)^{\Delta[1]_t}$ that is the closure of the cartesian product.
There is an evident injection $\operatorname{S}((B_i)^{\Delta[1]_t})\to(\operatorname{S}(B_i))^{\Delta[1]_t}$. 
By composing this and the morphism $(\operatorname{S}(B_i))^{\Delta[1]_t}\to(B_{i+1})^{\Delta[1]_t}$ induced by $\beta_i$, we obtain $\tilde{\beta}_i:\operatorname{S}((B_i)^{\Delta[1]_t})\to(B_{i+1})^{\Delta[1]_t}$, which defines a prestratified prespectrum $\{(B_i)^{\Delta[1]_t}, \tilde{\beta}_i\}$.

Since the morphisms $\Delta[0]\to\Delta[1]_t$ are elementary anodyne extensions and $B_i$ is fibrant, the induced morphism $(B_i)^{\Delta[1]_t}\to B_i\times B_i$ is an acyclic fibration.

We put $\operatorname{Path}(B)$ to be the mapping stratified stable simplicial set $B^{\Sigma_+^{\infty}\Delta[1]t}$, where $\Sigma_+^{\infty}\Delta[1]_t$ is the stratified stable simplicial set corresponding to $\Delta[1]_t$.
Then by construction, we have $\operatorname{Ps}(\operatorname{Path}(B))=\{(B_i)^{\Delta[1]_t}, \tilde{\beta}_i\}$.
The morphism $s:B\to \operatorname{Path}(B)$ corresponds to the constant path and it is a weak equivalence by the two out of three axiom. 
\end{proof}

\begin{prop} Let $A, B, C$ be fibrant stratified stable simplicial sets. 
For any morphism $A\to C$ and a fibration (resp. aspherical fibration) $B\to C$, the pullback $A\times_{C}B$ is again a fibrant stratified stable simplicial set and the projection $A\times_{C}B\to A$ is a fibration (resp. aspherical fibration).
\end{prop}

\begin{proof} By the definition of the functor $\operatorname{Ps}$, we have \[(A\times_{C}B)_i=A_i\times_{C_i}B_i\] for any $i\in\mathbb{N}$.
Since any fibration (resp. aspherical fibration) in Ozornova-Rovelli model structure is preserved by pullback, $(A\times_{C}B)_i\to A_i$ is a fibration for any $i\in\mathbb{N}$.
This completes the proof.
\end{proof}

\begin{cor}The full subcategory of fibrant stratified stable simplicial sets in $\mathbf{Set}_*^{t\Delta_{st}^{op}}$ admits a structure of a category of fibrant objects with the weak equivalences and fibrations in Definition \ref{wf}.
\end{cor}

\begin{rem}As is mentioned above, it is shown that the category of combinatorial spectra admits a model structure in \cite{B} by using the cofibrantly generated classical model structure on simplicial sets.
It might be possible to obtain the stratified analogue by using the cofibrantly generated model structure on $\mathbf{msSet}$ given in \cite{OR}.
\end{rem}

\begin{thm}
Suppose we have weak equivalences $f:X\to X'$ and $g:Y\to Y'$ of stratified stable simplicial sets. 
Then the morphism $f\oplus g:X\tilde{\wedge} Y\to X'\tilde{\wedge}Y'$ is also a weak equivalence.
\end{thm}

\begin{proof}
By definition, for $i, n\in\mathbb{N}$, we observe that $(X\tilde{\wedge}Y)_i([n])=\{(x, y)\in(X\tilde{\oplus} Y)([n-i+1])|d_1\cdots d_{n+1}(x, y)=*, d_j(x, y)=* (j\leq 0 \ {\rm or} \ n+2\leq j)\}$, where $d$'s are the structure morphisms of $X\tilde{\oplus} Y$.
Thus, $(X\tilde{\wedge}Y)_i$ is the join of two pointed stratified simplicial sets.
Combining this observation and Corollary \ref{1}, we obtain the desired result.
\end{proof}

\begin{rem}
We do not know whether $\tilde{\wedge}$ is a product on the full subcategory of fibrant stratified stable simplicial sets.

Let $\operatorname{F}:\mathbf{Set}_*^{t\Delta^{op}}\to\mathbf{Set}_*^{t\Delta^{op}}$ denote a fibrant replacement functor with respect to the pointed Ozornova-Rovelli model structure.
For any prestratified prespectrum $\{X_i, \xi_i\}$, we have fibrant pointed prestratified set $\operatorname{F}(X_i)$ for all $i\in\mathbb{N}$.
However there may not exsist the structure morphism $\operatorname{S}\wedge\operatorname{F}(X_i)\to\operatorname{F}(X_{i+1})$, which exists up to homotopy.
\end{rem}

\subsection{Reedy-like structure}
In this article, we only investigate prestratified stable simplicial objects in the category of (pointed) sets.
It would be worth studying such objects in other ($\infty$-)categories equipped with own homotopy theories.

It is well known that $\Delta$ is a Reedy category and is shown in \cite{OR} $t\Delta$ is also a Reedy category.
Moreover, it is shown there that these two categories have better properties, that is to say, they are regular skeletal.
In this section, we observe $\Delta_{st}$ and $t\Delta_{st}$ also have similar properties.

\begin{definition}[\cite{Hir}]A Reedy category is a category $\mathcal{R}$ equipped with two subcategories $\mathcal{R}_+$ and $\mathcal{R}_-$, both of which contain all the objects, and a total ordering on the set $\operatorname{ob}(\mathcal{R})$ of objects, defined by a degree function $\operatorname{deg}:\operatorname{ob}(\mathcal{R})\to\mathbb{N}$ such that:
\begin{itemize}
\item Every nonidentity morphism in $\mathcal{R}_+$ raises degree,
\item Every nonidentity morphism in $\mathcal{R}_-$ lowers degree, and
\item Every morphism $f$ in $\mathcal{R}$ factors uniquely as a morphism in $\mathcal{R}_-$ followed by a morphism in $\mathcal{R}_+$.
\end{itemize}
\end{definition}

As is well known, $\Delta$ is a Reedy category with the following structure:
\[\operatorname{deg}:\operatorname{ob}(\Delta)\to\mathbb{N}, [n]\mapsto n.\]
The subcategories $\Delta_+$ and $\Delta_-$ consist of injective maps and surjective maps respectively. 

\begin{thm}[\cite{OR}] The category $t\Delta$ is a Reedy category with the following structure:
The degree map $\operatorname{deg}:\operatorname{ob}(t\Delta)\to\mathbb{N}$ is given by
\[\operatorname{deg}([0])=0, \hspace{0.2em} \operatorname{deg}([k])=2k-1, \hspace{0.2em} \operatorname{deg}([k]_t)=2k, \hspace{0.2em} k\geq1.\]
The subcategory $t\Delta_+$ is generated by $\Delta_+$ and morphisms $\varphi:[n]\to[n]_t$ for all $n\geq1$ and $t\Delta_-$ is generated by $\Delta_-$ and morphisms $\zeta^i:[n+1]_t\to[n]$ for all $n\geq1$ and $0\leq i\leq n$.
\end{thm}

\begin{definition}[\cite{OR}]\label{regske}A Reedy category $\mathcal{R}$ is regular skeletal if the following conditions hold.

(1) Every morphism in $\mathcal{R}_-$ admits a section.

(2) Two parallel morphisms of $\mathcal{R}_-$ are equal if and only if they admit the same set of sections.

(3) Every morphism of $\mathcal{R}_+$ is a monomorphism.
\end{definition}

It is easy to show that $\Delta$ is regular skeletal.
Furthermore, Ozornova and Rovelli have shown the following.
\begin{thm}[\cite{OR}]The Reedy category $t\Delta$ is regular skeletal.
\end{thm}

In \cite{OR}, that $t\Delta$ admits is a regular skeletal Reedy category with this structure plays a pivotal role.
We will show that $t\Delta_{st}$ is {\it almost} a regular skeletal Reedy category with an analogous structure.
To do that, we show that $\Delta_{st}$ is also almost regular skeletal.

We consider the following structure on $\Delta_{st}$:
$\operatorname{deg}:\operatorname{ob}(\Delta_{st})\to\mathbb{Z}, [n]\mapsto n$
The subcategories $(\Delta_{st})_-$ (resp. $(\Delta_{st})_+$) is generated by identity morphisms and $s^i$'s (resp. identity morphisms and $d^i$'s).

\begin{lem}(1) For any $i$ and $n$, the morphism $s^i:[n+1]\to[n]$ in $\Delta_{st}$ is an epimorphism.

(2) For any $i$ and $n$, the morphism $d^i:[n-1]\to[n]$ in $\Delta_{st}$ is an monomorphism.
\end{lem}

\begin{proof}We prove (1). (2) can be proven by the same argument.
Assume that $\alpha s^i=\beta s^i$ for some morphisms $\alpha, \beta:[n]\to[m]$ in $\Delta_{st}$.
Then these morphisms $s^i$, $\alpha$, and $\beta$ are represented respectively by morphisms $s^i:[n+1+k]\to[n+k]$ and $\tilde{\alpha}, \tilde{\beta}:[n+k]\to[m+k]$ for some $k\in\mathbb{N}$ and $\tilde{\alpha} s^i=\tilde{\beta} s^i$ holds.
Since $s^i$ in $\Delta$ is an epimorphism, $\tilde{\alpha}=\tilde{\beta}$ in $\Delta$.
This shows that $\alpha=\beta$ in $\Delta_{st}$.
\end{proof}

\begin{lem}Every morphism in $\Delta_{st}$ factors uniquely as a morphism in $(\Delta_{st})_-$ followed by a mophism in $(\Delta_{st})_+$.
\end{lem}

\begin{proof}Recall that every morphism $f:[i]\to[j]$ in $\Delta$ is uniquely decomposed as $g\circ h:[i]\to[im(f)-1]\to[j]$ with $g$ a monomorphism and $h$ an epimorphism, where $im(f)$ denotes the cardinarity of the image of $f$.

Let $\theta:[m]\to[n]$ be a morphism in $\Delta_{st}$.
Then there exists a morphism $\tilde{\theta}:[m+k]\to[n+k]$ in $\Delta$ with a natural number $k\in\mathbb{N}$, which represents $\theta$.
The morphism $\tilde{\theta}$ in $\Delta$ can be written uniquely as $\tilde{\alpha}\circ\tilde{\beta}$  with $\tilde{\alpha}$ a monomorphism and $\tilde{\beta}$ an epimorphism in $\Delta$.
Then the morphisms $\alpha$ and $\beta$ in $\Delta_{st}$ represented by $\tilde{\alpha}$ and $\tilde{\beta}$ respectively are monomorphism and epimorphism respectively by the lemma above, and we have $\theta=\alpha\circ\beta$.
\end{proof}

Expect for that the codomain of the degree map is $\mathbb{Z}$, $\Delta_{st}$ satisfies the all requirements for being a regular skeletal Reedy category.

\begin{lem}For each $n\in\mathbb{Z}$, the morphism $\varphi:[n]\to[n]_t$ in $t\Delta_{st}$ is a monomorphism and an epimorphism.
\end{lem}
\begin{proof}This follows from that $\varphi:[n]\to[n]_t$ in $t\Delta$ is a monomorphism and an epimorphism proven in \cite{OR}.

\end{proof}

We consider the following structure on $t\Delta_{st}$:
A map $\operatorname{deg}:\operatorname{ob}(t\Delta)\to\mathbb{Z}$ given by
\[\operatorname{deg}([0])=0, \operatorname{deg}([k])=2k-1, \operatorname{deg}([k]_t)=2k\]
for $k\neq 0$.
The subcategory $(t\Delta_{st})_+$ is generated by $(\Delta_{st})_+$ and morphisms $\varphi:[n]\to[n]_t$ for all $n\geq1$ and $t\Delta_-$ is generated by $\Delta_-$ and morphisms $\zeta^i:[n+1]_t\to[n]$ for all $n\geq1$ and $0\leq i\leq n$.

The arguments in \cite[Proposition C.4]{OR} and above show the following.

\begin{prop}$t\Delta_{st}$ satisfy the conditions (1), (2) and (3) in Definition \ref{regske}.
\end{prop}

\end{document}